\documentclass[10pt,reqno,a4paper]{amsart}

\usepackage[hmargin=2.5cm,bmargin=2.5cm,tmargin=3cm]{geometry}

\usepackage{amsmath,amssymb,amsthm,amsfonts,mathrsfs}
\usepackage{pgf,tikz}
\usepackage{graphicx}
\usepackage{hyperref}

\usepackage[normalem]{ulem}
\title[]{On domain monotonicity of Neumann eigenvalues of convex domains}

\author[P. Freitas]{Pedro Freitas}
\author[J.B. Kennedy]{James B.\ Kennedy}

\address{Grupo de F\'{\i}sica Matem\'{a}tica, Instituto Superior T\'{e}cnico, Universidade de Lisboa, Av. Rovisco Pais, 1049-001 Lisboa, Portugal}
\email{pedrodefreitas@tecnico.ulisboa.pt}
\address{Departamento de Matem\'atica, Faculdade de Ci\^encias, Universidade de Lisboa, Campo Grande, Edif\'icio C6, P-1749-016 Lisboa, Portugal {\rm and} Grupo de F\'isica M\'atematica, Instituto Superior T\'{e}cnico, Universidade de Lisboa, Av. Rovisco Pais, 1049-001 Lisboa, Portugal}
\email{jbkennedy@ciencias.ulisboa.pt}
\subjclass[2010]{}

\newtheorem{theorem}{Theorem}[section]
\newtheorem{lemma}[theorem]{Lemma}
\newtheorem{proposition}[theorem]{Proposition}

\newtheorem{thm}[theorem]{Theorem}
\newtheorem{prop}[theorem]{Proposition}

\theoremstyle{definition}

\newtheorem{problem}[theorem]{Question}

\theoremstyle{remark}
\newtheorem{remark}[theorem]{Remark}
\newtheorem{example}[theorem]{Example}

\numberwithin{equation}{section}
\numberwithin{figure}{section}

\newcommand{\R}{\mathbb{R}}
\newcommand{\N}{\mathbb{N}}

\DeclareMathOperator{\diam}{diam}

\newcommand{\bo}{{\rm O}}

\newcommand{\ds}{\displaystyle}

\newcommand{\eqskip}{ \vspace*{2mm}\\ }
\newcommand{\fr}[2]{\frac{\ds #1}{\ds #2}}

\newcommand{\txtb}{\textcolor{black}}

\newcommand{\optconst}[2]{{\alpha_{#1,#2}}}
\newcommand{\genoptconst}[2]{{\beta_{#1,#2}}}

\begin{document}

\date{\today}

\begin{abstract}
Inspired by a recent result of Funano's, we provide a sharp quantitative comparison result between the first nontrivial
eigenvalues of the Neumann Laplacian on bounded convex domains $\Omega_{1} \subset \Omega_{2}$ in any dimension $d$ greater than
or equal to two, recovering domain monotonicity up to an explicit multiplicative factor. We provide upper and lower bounds for such
multiplicative factors for higher-order eigenvalues, and study their behaviour with respect to the dimension and order. We further
consider different scenarios where convexity is no longer imposed. In a final section we formulate some related open problems.
\end{abstract}
\maketitle

\section{Introduction}

It is a basic and well-known fact in the spectral theory of partial differential equations that the eigenvalues of the Dirichlet Laplacian on Euclidean domains satisfy a
domain monotonicity principle, while their Neumann counterparts do not.

More precisely, for a bounded, connected, sufficiently regular open set $\Omega \subset \R^d$, $d\geq 2$, denote by
$0<\lambda_1 (\Omega) < \lambda_2 (\Omega) \leq \lambda_3 (\Omega) \leq \ldots $ the ordered eigenvalues, repeated according to their multiplicities, of the Dirichlet problem
\begin{equation}
\label{eq:dirichlet}
\begin{aligned}
	-\Delta u &=\lambda u \qquad &&\text{in } \Omega,\\
	u &=0 \qquad &&\text{on } \partial\Omega,
\end{aligned}
\end{equation}
then, if $\Omega_1 \subset \Omega_2$, the variational (min-max) characterisation of the eigenvalues and the canonical identification of $H^1_0 (\Omega_1)$ as a subset
of $H^1_0 (\Omega_2)$ immediately imply that $\lambda_k (\Omega_1) \geq \lambda_k (\Omega_2)$ for all $k \in \N$.

But if $0 = \mu_0 (\Omega) < \mu_1 (\Omega) \leq \mu_2 (\Omega) \leq \ldots $ denote the corresponding eigenvalues of the Neumann problem
\begin{equation}
\label{eq:neumann}
\begin{aligned}
	-\Delta u &=\mu u \qquad &&\text{in } \Omega,\\
	\frac{\partial u}{\partial\nu} &=0 \qquad &&\text{on } \partial\Omega,
\end{aligned}
\end{equation}
where $\nu$ is the outer unit normal to $\Omega$, then there are simple examples of $\Omega_1 \subset \Omega_2$ such that $\mu_k (\Omega_1) < \mu_k (\Omega_2)$ for
some $k \in \N$, even if both domains are convex. The classical example is to take $\Omega_2$ to be a square, say of side length $1/\sqrt{2}$, and $\Omega_1$ to
be a line segment of length equal to that of the diagonal of the square, here $1$ (or else a very thin rectangle of longer side length $1-\varepsilon$); then up to isometries
$\Omega_1 \subset \Omega_2$, and $\mu_1 (\Omega_1) = \pi^2 < \mu_1 (\Omega_2) = 2\pi^2$.

What is far less well known is that, if $\Omega_1 \subset \Omega_2$ and both domains are convex, then it is possible to recover a monotonicity statement for the Neumann
eigenvalues up to a (dimensional) factor: very recently, Funano proved the following theorem in \cite{f22}, building on an earlier paper~\cite{f16}. \txtb{In our notation:}

\begin{theorem}[\cite{f22}, Theorem~1.1]\label{tf22}
There exists a \txtb{universal}  constant $C>0$ such that, for any bounded convex domains $\Omega_1 \subset \Omega_2 \subset \R^d$, $d \geq 2$, and any $k \in \N$,
\begin{equation}
\label{eq:funano}
	\frac{\mu_k (\Omega_1)}{\mu_k (\Omega_2)} \geq \frac{C}{d^2}.
\end{equation}
\end{theorem}

\noindent\txtb{As indicated in~\cite{f22}, the constant $C$ in the result above may be taken to equal $1/92^2$.}
Our main goal here is to make a number of observations on the value of the optimal constant, that is, to study the quantities
\begin{equation}
\label{eq:inf}
	\optconst{k}{d} := \inf \left\{\frac{\mu_k (\Omega_1)}{\mu_k (\Omega_2)} : \Omega_1\subset\Omega_2
	\subset \R^d \text{ bounded, convex domains}\right\}
\end{equation}
for $k,d\in \N$, which may be thought of as measuring the degree to which domain monotonicity for Neumann inclusion of convex domains may fail. Funano's result may then be reformulated
as asserting the existence of a universal constant $C>0$ such that
\begin{equation}
\label{eq:inf-funano}
	\frac{C}{d^2} \leq \optconst{k}{d} \leq 1
\end{equation}
for all $k,d \in \N$. It is not hard to find examples which show that $\alpha_{k,d} < 1$ for all $k\in\N$ and $d \geq 2$, see Proposition~\ref{prop:bounds-behaviour}(1), that is, for any given eigenvalue in any given dimension $d\geq 2$ domain monotonicity never holds among all bounded convex domains. (Note that the corresponding supremum is trivially always $\infty$, as we may fix $\Omega_{2}$ and consider $\Omega_{1}$ to be a ball of arbitrarily small radius.)

We shall be interested in $\optconst{k}{d}$ both for small values of $k$ (in particular $k=1$) and the asymptotic behaviour as $k \to \infty$ for fixed $d$, as well as the behaviour of $\optconst{k}{d}$
as a function of $d$ for fixed $k$.

Our motivation for the case $k=1$ comes from the observation that among most convex domains with explicitly computable $\mu_1$, the example mentioned above, where
$\Omega_2$ is a square and $\Omega_1$ is a line segment equal to a diagonal of the square, appears to minimise the ratio $\frac{\mu_1 (\Omega_1)}{\mu_1 (\Omega_2)}$.
We will see, however, that this is not actually the infimal value; in fact, it turns out that the optimal value
can be obtained directly by combining, and examining more closely, classical estimates by Payne--Weinberger~\cite{pw60} (see also~\cite{bebe}) and Kr\"{o}ger~\cite{k99}
(see also~\cite{hm}) on $\mu_1(\Omega)$
in terms of the diameter of the convex domain $\Omega$:
%We will address this case in Section~\ref{sec:k=1}, where we prove (see Theorem~\ref{thm:k=1}) that:
\begin{thm}
\label{scaledmonot1}
The constant $\optconst{1}{d}$ defined by~\eqref{eq:inf} \txtb{is given by}
\begin{equation}
\label{eq:k=1}
	\optconst{1}{d} = \fr{\pi^2}{4j_{\fr{d}{2}-1,1}^2},
\end{equation}
where $j_{\nu,1}$ is the first zero of the Bessel function $J_\nu$ of the first kind of order $\nu$. In particular, the sharp inequality
\begin{equation}
\label{eq:scaledmonot1}
	\mu_{1}\left(\Omega_{1}\right) \geq \fr{\pi^2}{4j_{\fr{d}{2}-1,1}^2} \mu_{1}\left(\Omega_{2}\right)
\end{equation}
holds between the first non-trivial eigenvalues of any two bounded convex domains $\Omega_1 \subset \Omega_2 \subset \R^d$, $d\geq 2$.
\end{thm}

%{\color{blue}
%\begin{remark}
% Using the bound $j_{\nu,k}\leq (k+\nu/2-1/4)\pi$, valid for $\nu\geq 1/2$~\cite{heth}, inequality~\eqref{eq:k=1} yields
% \[ \mu_{1}\left(\Omega_{1}\right) \geq \fr{4}{(d+1)^2}\mu_{1}\left(\Omega_{2}\right)  \;\;(d\geq 3) \]
%  which, while not sharp, has the appeal of simplicity.
%\end{remark}
%}

See Section~\ref{sec:k=1} for the proof, including a detailed description of domains $\Omega_{1,d} \subset \Omega_{2,d}$ attaining the infimum
in the degenerate limit (essentially, $\Omega_{1,d}$ will be a line segment, or long thin parallelepiped, and $\Omega_{2,d}$ will be a thin double
cone), which play a central role. There, for comparison, we also compare the line segment with a number of other domains with explicitly computable
first eigenvalue. We also note in passing the simpler (but non-sharp) inequality
\begin{equation}
\label{eq:scaledmonot1-simple}
	\mu_{1}\left(\Omega_{1}\right) \geq \fr{\pi^2}{2d(d+4)}\mu_{1}\left(\Omega_{2}\right)  \qquad (d\geq 2),
\end{equation}
as follows from \eqref{eq:scaledmonot1} and the bound $j_{\nu,1}^2\leq 2(\nu+1)(\nu+3)$~\cite[p. 486]{wats}.

Now since the constant in Funano's result is valid for all eigenvalues and any dimension at least two, Theorem~\ref{scaledmonot1}
immediately provides a bound for the constant $C$ in Theorem~\ref{tf22}, namely
\begin{equation}
\label{eq:kroeger-funano}
	C \leq \fr{ \pi^2 }{ j_{0,1}^2} \approx 1.7066.
\end{equation}
Actually, if one compares the optimal domains in Kr\"{o}ger's upper bound for higher values of $k$ (and all $d$) with the corresponding eigenvalue of
a segment of the same diameter, one can obtain a complementary upper bound to \eqref{eq:inf-funano} of the form
\begin{equation}
\label{eq:funano-upper}
	\optconst{k}{d} \leq \frac{C_k}{d^2}
\end{equation}
for all $d \geq 2$, where $C_k>0$ depends only on $k\in \N$, which confirms that the correct dimensional behaviour is $\optconst{k}{d} \sim d^{-2}$ for
any fixed $k \in \N$. See Section~\ref{sec:asymptotics}, and in particular Proposition~\ref{prop:bounds-behaviour}, where an explicit upper bound is provided. 
By way of comparison, when $k=1$, \eqref{eq:k=1} yields the explicit asymptotic behaviour
\begin{equation}
	\optconst{1}{d} = \fr{\pi^2}{d^2} + \bo\left(d^{-3}\right), \mbox{ as } d\to\infty.
\end{equation}
We also show in Section~\ref{sec:asymptotics} that $\optconst{k}{d}$ is in fact a decreasing function of the dimension for each fixed $k\geq 1$
(Theorem~\ref{thm:dimension-monotonicity}).

The question about the asymptotics in $k$ is largely motivated by the Weyl asymptotics for the Neumann eigenvalues~\cite[pp 31 ff]{chav}, which immediately implies
the following result:
\begin{proposition}\label{static}
Suppose $\Omega_1 \subset \Omega_2 \subset \R^d$, $d \geq 2$, are \emph{any fixed} bounded domains with sufficiently regular boundary. Then
\begin{displaymath}
	\lim_{k \to \infty} \frac{\mu_k (\Omega_1)}{\mu_k (\Omega_2)} = \left(\fr{\left|\Omega_{2}\right|}{\left|\Omega_{1}\right|}\right)^{2/d}\geq 1,
\end{displaymath}
with equality if and only if $|\Omega_1| = |\Omega_2|$ (and thus $\Omega_1 = \Omega_2$).
\end{proposition}
This, in turn, raises a number of questions. Although the only relevant point in the limit in Proposition~\ref{static} is the relation between the volumes
of the two sets involved, one can still consider the behaviour of the ratio in different settings.

By way of example, in Section~\ref{sec:polya} we will consider the special case of the ratio $\frac{\mu_k(\Omega)}{\mu_k(\Omega_k)}$ where
$\Omega \subset \Omega_k$ for all $k \in \N$ is a fixed domain, and no convexity restrictions are imposed on either $\Omega$ or the family of domains $\Omega_k$. We show
(Theorem~\ref{thm:fixed-inner-domain}) that for any bounded, sufficiently regular domains $\Omega \subset \Omega_k$, we necessarily recover
the asymptotic monotonicity
\begin{displaymath}
	\liminf_{k \to \infty} \frac{\mu_k (\Omega)}{\mu_k (\Omega_k)} \geq 1;
\end{displaymath}
the stronger, volume-adjusted version of this inequality, namely
\begin{displaymath}
	\liminf_{k\to\infty} \frac{\mu_{k}(\Omega)|\Omega|^{2/d}}{\mu_{k}(\Omega_{k})|\Omega_k|^{2/d}} \geq 1
\end{displaymath}
under the same assumptions, is actually equivalent to P\'olya's conjecture for the Neumann Laplacian, as a consequence of the superadditivity of the sequence of maximal
Neumann eigenvalues for domains of fixed volume proved by Colbois and El Soufi \cite{colels}; see Remark~\ref{rem:colels}.

The other fundamental question is whether the ``asymptotic domain monotonicity'' of Proposition~\ref{static}, valid for any fixed pair of domains, can be made to be uniform over
all bounded convex domains, that is, whether $\optconst{k}{d} \to 1$ as $k \to \infty$, for any fixed $d \geq 2$. This, and other natural (but likely difficult) questions and
open problems, are collected in Section~\ref{sec:open}.

\section{The optimal constant for the first nonzero eigenvalues}
\label{sec:k=1}

We start with the case $k=1$, and in particular give the proof of Theorem~\ref{scaledmonot1}. We first describe the sequences of domains mentioned in the introduction which lead
to the correct value of $\optconst{1}{d}$ in \eqref{eq:k=1} in the degenerate limit: we form $\Omega_{2,n}$ by gluing two flat-bottomed finite spherical cones each of given height
$D/2>0$ and of opening angle shrinking to zero, along their flat bases to form a double
cone of fixed diameter $D$, the distance between its two vertices; $\Omega_1$ is a line segment of length $D>0$ independent of $n$ (or,
among Lipschitz domains, one may take a sequence of long, thin right parallelepipeds $\Omega_{1,n}$ approaching the segment).

In dimension three and above, these domains $\Omega_{2,n}$ were identified by Kr\"{o}ger \cite[Remark~2]{k99} as the optimisers of his diameter bound, but he did not provide the details.
In the proof of Theorem~\ref{scaledmonot1} we will prove the convergence of $\mu_1(\Omega_{2,n})$ to the claimed value, and also show that the same type of domain works in
dimension two, where $\Omega_{2,n}$ is now a rhombus (see Figure~\ref{fig:rhombus}).

But first, for the sake of comparison, in the following table we give the values of $\mu_1(\Omega)$ for various planar domains with diameter normalised to $2$ (so the
circle has radius $1$, the square has side length $\sqrt{2}$, and so on). Kr\"{o}ger's upper bound \cite[Theorem~1]{k99} and Payne--Weinberger's lower
bound~\cite[Eq.~(1.9)]{pw60} in terms of the line segment of the same diameter are included for comparison.
\begin{center}
\begin{tabular}{l|l|l}
	$\Omega$/bound & value of $\mu_1(\Omega)$ & $\mu_1(\text{line segment})/\mu_1(\Omega)$\\ \hline
	optimal bound (Kr\"{o}ger/degenerating rhombi) & $j_{0,1}^2 \approx 5.783$ & $\approx 0.427$\\ %\eqskip
	square & $\pi^2/2 \approx 4.935$ & $0.5$\\ %\eqskip
	optimal sector & $\approx (j_{\pi/1.654,1}') ^2\approx 4.67$ & $\approx 0.53$\\ %\eqskip
	equilateral triangle \cite{mccartin} & $4\pi^2/9 \approx 4.386$ & 0.5625\\ %\eqskip
	Reuleaux triangle  & $\approx 3.487$ & $\approx 0.707$\\ %\eqskip \cite{ap}
	disk & $(j_{0,1}') ^2 \approx 3.39$ & $\approx 0.73$\\ %\eqskip
	line segment & $\pi^2/4 \approx 2.467$ & $1$
\end{tabular}
\end{center}
%The eigenvalues of the equilateral triangle are of the form
%\begin{displaymath}
%	\frac{4}{27} \frac{\pi^2}{r^2} (m^2+mn+n^2) = \frac{16\pi^2}{9D^2} (m^2+mn+n^2),
%\end{displaymath}
%where $r=D/2\sqrt{3}$ is the inradius of the triangle (and $D$ its sidelength and diameter), and $m,n \in \N_0$; $m=n=0$ corresponds to the ground state energy and $m=0$, $n=1$ (or vice versa) to $\mu_1$. If $D=2$, this yields a value of $\frac{4\pi^2}{9}$.

\begin{figure}[ht]
\begin{center}
\begin{minipage}[c]{0.5\linewidth}
\includegraphics[scale=0.63]{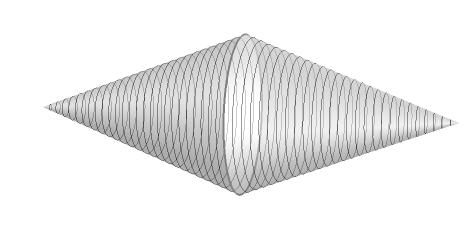}
\end{minipage}\hfill
\begin{minipage}[c]{0.5\linewidth}
\begin{tikzpicture}[scale=0.9]
\draw[thick] (-4,0) -- (0,1.5) -- (4,0) -- (0,-1.5) -- (-4,0);
\draw[thick,dotted] (0,1.5) -- (0,-1.5);
\draw[thick,dashed] (4,0) -- (-4,0);
\draw[thick] (-3.5,0) arc (0:21:0.5);
\draw[thick] (-3.5,0) arc (0:-21:0.5);
\draw[thick] (-3.45,0) arc (0:21:0.55);
\draw[thick] (-3.45,0) arc (0:-21:0.55);
%\node at (-3.5,0) [anchor=west] {$\theta$};
\node at (-3.75,0.1) [anchor=south] {$2\theta$};
\draw[thick] (3.5,0) arc (180:201:0.5);
\draw[thick] (3.5,0) arc (180:159:0.5);
\draw[thick] (3.45,0) arc (180:201:0.55);
\draw[thick] (3.45,0) arc (180:159:0.55);
\end{tikzpicture}
\end{minipage}
\caption{\txtb{The domain $\Omega_{2,n}$ in dimension $d=3$ (left) and in dimension $d=2$ (right). The latter also represents a suitable two-dimensional cross-section of $\Omega_{2,n}$ if $d\geq 3$. The dashed line has length $D$.}}
\label{fig:rhombus}
\end{center}
\end{figure}

\begin{proof}[Proof of Theorem~\ref{scaledmonot1}]
Suppose $\Omega_1 \subset \Omega_2 \subset \R^d$ are convex domains, with respective diameters $D_1 := \diam (\Omega_1) \leq \diam (\Omega_2) =: D_2$. That
$\optconst{1}{d} \geq \pi^2/4j_{\frac{d}{2}-1,1}^2$ follows immediately from combining the theorem of Payne--Weinberger \cite{pw60},
\begin{displaymath}
	\mu_1 (\Omega_1) \geq \frac{\pi^2}{D_1^2} \geq \frac{\pi^2}{D_2^2},
\end{displaymath}
with the bound of Kr\"{o}ger \cite[Theorem~1]{k99} with $m=1$,
\begin{displaymath}
	\mu_1 (\Omega_2) \leq \frac{4j_{\frac{d}{2}-1,1}^2}{D_2^2} = \frac{4j_{\frac{d}{2}-1,1}^2}{\pi^2}\times \frac{\pi^2}{D_2^2}
	\leq \frac{4j_{\frac{d}{2}-1,1}^2}{\pi^2} \mu_1 (\Omega_1).
\end{displaymath}
To prove sharpness we show that for the domains $\Omega_{2,n} \subset \R^d$ described above, \txtb{that is, double cones of angle of opening $2\theta = 2\theta(n) \to 0$ and constant diameter $D>0$, as depicted in Figure~\ref{fig:rhombus},} we do in fact have
\begin{equation}
\label{eq:cone-limit}
	\mu_1 (\Omega_{2,n}) \longrightarrow \frac{4j_{\frac{d}{2}-1,1}^2}{D^2} = \frac{4j_{\frac{d}{2}-1,1}^2}{\pi^2}\mu_1 (\Omega_1),
\end{equation}
where $\Omega_1 \subset \Omega_{2,n}$ is the line segment of length $D$.

To this end, take any hyperplane $\Pi$ \txtb{bisecting} $\Omega_{2,n}$ passing through both vertices ($\Pi$ is represented by the dashed line in Figure~\ref{fig:rhombus}-right) and note that, up to the correct choice of basis of eigenfunctions, all eigenfunctions of \eqref{eq:neumann} are either symmetric (even) or antisymmetric (odd) with respect to this plane; more precisely, assuming that $\Pi = \{x_d = 0\}$ and denoting by $\Phi : (x',x_d) \mapsto (x',-x_d)$ the even reflection mapping leaving $\Pi$ invariant (where we have written $(x',x_d) \in \R^{d-1} \times \R$), we may choose the eigenfunctions $\psi_k$, $k \in \N$, of \eqref{eq:neumann} in such a way that, for each $k\in\N$, either $\Phi(\psi_k) = \psi_k$ (the symmetric case) or $\Phi(\psi_k) = -\psi_k$ (the antisymmetric case).

\begin{figure}[ht]
\begin{center}
\begin{minipage}[c]{0.5\linewidth}
\begin{tikzpicture}[scale=0.8]
\draw[thick] (-4,0) -- (0,1.5) -- (4,0) -- (0,-1.5) -- (-4,0);
\draw[thick,dashed] (0,1.5) -- (0,-1.5);
\node at (-1.5,0) {$\psi > 0$};
\node at (1.5,0) {$\psi < 0$};
\end{tikzpicture}
\end{minipage}\hfill
\begin{minipage}[c]{0.5\linewidth}
\begin{tikzpicture}[scale=0.8]
\draw[thick] (-4,0) -- (0,1.5) -- (4,0) -- (0,-1.5) -- (-4,0);
\draw[thick,dashed] (4,0) -- (-4,0);
\node at (0,0.6) {$\psi > 0$};
\node at (0,-0.6) {$\psi < 0$};
\end{tikzpicture}
\end{minipage}
\caption{\txtb{The nodal pattern for the first eigenfunction $\psi$ of $\Omega_{2,n}$ associated with the first nontrivial symmetric (left) and antisymmetric (right) eigenvalues with respect to the plane $\Pi$. In each case the dashed line indicates the zero (nodal) set of the eigenfunction.}}
\label{fig:nodal}
\end{center}
\end{figure}

We always have that $\mu_1(\Omega_{2,n})$ is the first Laplacian eigenvalue, call it $\tau_1$, of each of the two nodal domains (with a Neumann condition on the exterior boundary, represented by the solid lines, and a Dirichlet condition on the interior boundary, represented by the dashed line), of whichever eigenfunction it corresponds to.

Now it follows from Lemma~\ref{lem:nodal} below, applied to either of the nodal domains in the antisymmetric case, that for $\theta = \theta (n)$ sufficiently small the
eigenfunction associated with $\tau_1$ is symmetric, since Lemma~\ref{lem:nodal} shows that the first antisymmetric eigenvalue diverges at least as fast as
$\sin^{-2} (\theta)$ as $\theta \to 0$. So it suffices to study the symmetric case. Denote by $\Omega_n^\pm$ the two corresponding nodal domains, which are the two flat
bottomed finite cones constituting $\Omega_{2,n}$, and by $\tau (\Omega_n^+) = \mu_1(\Omega_{2,n})$ the eigenvalue with mixed boundary conditions.

We will use a simple domain monotonicity argument to compare this eigenvalue with the first eigenvalue of two sectors (cones with spherical caps rather than flat bases) whose first eigenvalue is known explicitly. More precisely, let $\tau_{1}(C_{R,\theta})$ denote the eigenvalue of the finite cone obtained by intersecting a ball of radius $R$ with an infinite circular cone of angle of opening $2\theta$ with apex at the centre of the ball, with Neumann boundary conditions along the side of the cone and Dirichlet conditions on the spherical cap. This eigenvalue
coincides with the first Dirichlet eigenvalue of the ball with the same radius, as the corresponding eigenfunction  is radial.

Then we have the (set) inclusions $C_{\frac{D}{2},\theta} \subset \Omega_n^+ \subset C_{\frac{D}{2\cos(\theta)}}$, see Figure~\ref{fig:comparison}, which, due to the presence of Dirichlet conditions on the base of the cone and the spherical caps, translate directly to the reverse chain of inequalities for the respective form domains. The variational characterisation of the respective eigenvalues (cf.\ the proof of Lemma~\ref{lem:nodal}) thus yields the two-sided bound
\begin{equation}
\label{eq:cone-squeeze}
	\frac{4 \cos^2(\theta)j_{d/2-1,1}^2}{D^2} = \tau_{1}\left(C_{\frac{D}{2\cos(\theta)},\theta}\right)\leq \tau (\Omega_n^+) = \mu_{1}(\Omega_{2,n})
	\leq \tau_{1}\left(C_{\frac{D}{2},\theta}\right) =  \fr{4 j_{d/2-1,1}^2}{D^2},
\end{equation}
for all $\theta$ sufficiently small (equivalently, $n$ sufficiently large) that the eigenfunction associated with $\mu_{1}$ is indeed symmetric, as described above.
\begin{figure}[ht]
\begin{center}
\begin{tikzpicture}
\filldraw[lightgray] (-4,0) -- (0,1.5) -- (0,-1.5) -- (-4,0);
\draw[thick] (-4,0) -- (0,1.5);
\draw[thick] (0,-1.5) -- (-4,0);
\draw[thick,dashed] (0,1.5) -- (0,-1.5);
\draw[thick,dotted] (-4,0) -- (0.275,0);
\draw[thick] (-3.5,0) arc (0:21:0.5);
\draw[thick] (-3.5,0) arc (0:-21:0.5);
\draw[thick] (-3.45,0) arc (0:21:0.55);
\draw[thick] (-3.45,0) arc (0:-21:0.55);
%\node at (-3.5,0) [anchor=west] {$\theta$};
\node at (-3.75,0.1) [anchor=south] {$2\theta$};
\draw[thick,dashed] (0,0) arc (0:21:3.7);
\draw[thick,dashed] (0,0) arc (0:-21:3.7);
\draw[thick,dashed] (0.275,0) arc (0:21:4.2);
\draw[thick,dashed] (0.275,0) arc (0:-21:4.2);
%\node at (0.275,0) [anchor=west] {$\frac{D}{2\cos(\theta)}$};
\node at (-0.65,1.4) [anchor=east] {$C_{\frac{D}{2},\theta}$};
\node at (0.2,1.1) [anchor=west] {$C_{\frac{D}{2\cos(\theta)},\theta}$};
\end{tikzpicture}
\caption{\txtb{In dimension $d=2$, the two conical sectors $C_{\frac{D}{2\cos(\theta)},\theta}$ and $C_{\frac{D}{2},\theta}$ which bound the nodal domain $\Omega_n^+$ of the eigenfunction for $\mu_1(\Omega_{2,n})$ indicated in the figure by the triangle shaded light grey, of height $\frac{D}{2}$. The horizontal dotted line has total length $\frac{D}{2\cos(\theta)}$. For each of the three domains, Neumann conditions are imposed on the flat sides represented here by solid lines, and Dirichlet conditions are represented by dashed lines/arcs.}}
\label{fig:comparison}
\end{center}
\end{figure}

Letting $n \to \infty$ and thus $\theta \to 0$ in \eqref{eq:cone-squeeze} yields the convergence claimed in \eqref{eq:cone-limit}.
\end{proof}

\begin{lemma}
\label{lem:nodal}
Suppose $\Omega \subset \R^d$ is a bounded Lipschitz domain which takes the form
\begin{displaymath}
	\Omega = \{ (x',x_d) \in \R^{d-1}\times \R : x' \in \Omega',\, 0<x_d<F(x') \}
\end{displaymath}
for some bounded Lipschitz domain $\Omega' \subset \R^{d-1}$ and some Lipschitz function $F: \overline{\Omega'} \to [0,\infty)$ such that $F(x')=0$ for all $x' \in \partial\Omega'$.

Let $M = \max \{F(x'):x' \in \Omega'\}$ and denote by $\txtb{\tau_1}(\Omega)$ the first eigenvalue of the Laplacian on $\Omega$ with Dirichlet conditions on $\Omega'$
(which in a slight abuse of notation we have identified with the set $\Omega'\times \{0\} \subset \partial\Omega \subset \R^d$) and Neumann conditions on the rest of 
$\partial\Omega$. Then
\begin{displaymath}
	\txtb{\tau_1} (\Omega) \geq \frac{\pi^2}{4M^2}.
\end{displaymath}
\end{lemma}

That is, $\txtb{\tau_1} (\Omega)$ is bounded from below by the first eigenvalue of an interval of length $M$ with mixed Dirichlet-Neumann conditions (equivalently, of the cylinder $\Omega' \times (0,M)$ with Dirichlet conditions on $\Omega'\times \{0\}$ and Neumann conditions elsewhere).

\begin{proof}
Let $\psi$ be a positive eigenfunction associated with $\txtb{\tau_1}(\Omega)$; then by standard regularity theory $\psi = 0$ identically on $\Omega'$, and $\psi$
restricted to each vertical line segment $[(x',0),(x',F(x'))]$ is a $H^1$-function on that line segment, for almost every $x' \in \Omega'$, and
\begin{displaymath}
	\txtb{\tau_1} (\Omega) = \frac{\int_\Omega |\nabla \psi|^2\,\textrm{d}x}{\int_\Omega |\psi|^2\,\textrm{d}x}
	\geq \frac{\int_\Omega |\psi_{x_d}|^2\,\textrm{d}x}{\int_\Omega |\psi|^2\,\textrm{d}x}
	=\frac{\int_{\Omega'}\int_0^{F(x')} |\psi_{x_d}|^2\,\textrm{d}x_d\,\textrm{d}x'}{\int_{\Omega'}\int_0^{F(x')} |\psi|^2\,\textrm{d}x_d\,\textrm{d}x'}.
\end{displaymath}
Using that
\begin{displaymath}
	\frac{\int_0^{F(x')} |\psi_{x_d}(x',x_d)|^2\,\textrm{d}x_d}{\int_0^{F(x')} |\psi(x',x_d)|^2\,\textrm{d}x_d}
	\geq \inf_{\substack{u \in H^1(0,F(x'))\\u(0)=0}} \frac{\int_0^{F(x')} |u'(t)|^2\,\textrm{d}t}{\int_0^{F(x')} |u(t)|^2\,\textrm{d}t} \geq \frac{\pi^2}{4F(x')^2}
	\geq \frac{\pi^2}{4M^2}
\end{displaymath}
for almost every $x' \in \Omega'$ leads to
\begin{displaymath}
	\txtb{\tau_1} (\Omega) \geq \fr{\int_{\Omega'}\frac{\pi^2}{4M^2} \int_0^{F(x')} |\psi|^2\,\textrm{d}x_d\,\textrm{d}x'}
	{\int_{\Omega'}\int_0^{F(x')} |\psi|^2\,\textrm{d}x_d\,\textrm{d}x'}
\end{displaymath}
and hence the claim.
\end{proof}

\section{Bounds on optimal constants for higher eigenvalues}
\label{sec:asymptotics}

For higher values of $k$, as mentioned in the introduction, it is easy to obtain a nontrivial upper bound on $\optconst{k}{d}$ comparing a line segment of length $D$
 with Kr\"{o}ger's diameter-based upper bounds; the values obtained in this way are
\begin{equation}
\label{eq:inf-kroeger}
	\optconst{k}{d} \leq c(k,d):= \begin{cases} \fr{\pi^2 k^2}{(2j_{0,1} + \pi(k-1))^2} \qquad &\text{if } d=2,\eqskip
	\fr{\pi^2 k^2}{4 j_{d/2-1,(k+1)/2}^2} \qquad & \text{if $d\geq 3$ and $k$ odd,}\eqskip
	\fr{\pi^2 k^2}{\left( j_{d/2-1,k/2}+ j_{d/2-1,k/2+1}\right)^2} \qquad & \text{if $d \geq 3$ and $k$ even.}\end{cases}
\end{equation}
These bounds are all nontrivial; moreover, as also intimated in the introduction, \eqref{eq:inf-kroeger} actually already gives the correct asymptotic power of $d$ for fixed $k$,
and for any fixed $d$ also behaves as might naturally be expected as $k \to \infty$, in light of Proposition~\ref{static}. 
In fact, these may be quite close to, or even be plausible candidates for, the actual values of $\alpha_{k,d}$, although this is also related to whether these values are attained
for pairs of domains approaching a line segment in a specific way (see Open Problem~\ref{question:optima}).
% \txtb{They even appear to be plausible candidates for the actual values of $\alpha_{k,d}$, although whether they are equal,
% or indeed how close they are, is also related to whether these are attained for pairs of domains approaching a line segment in a specific way (see Open Problem~\ref{question:optima})}.
%How close these bounds are to the actual values of $\alpha_{k,d}$, \txtb{or indeed whether they are actually equal,} is also related to whether, in general, each of these is attained for pairs of domains approaching a line segment in a specific way \txtb{(see Open Problem~\ref{question:optima})}.

\begin{prop}
\label{prop:bounds-behaviour}
The bounds $c(k,d)$ appearing on the right-hand side of\eqref{eq:inf-kroeger} satisfy:
\begin{enumerate}
\item $c(k,d) < 1$ for all $k \in \N$ and $d\geq 2$;
\item For each fixed $d\geq 2$, $c(k,d) \to 1$ as $k \to \infty$;
\item For each fixed $k \in \N$, there exists a constant $C_k > 0$ such that $c(k,d) \leq \frac{C_k}{d^2}$ for all $d \geq 2$.
\end{enumerate}
\end{prop}
\begin{proof}
 In dimension two, clearly the limit is one, and to see that the given expression is always below that value just notice that this statement is equivalent to
 $\pi< 2j_{0,1} \approx 4.81$. For $d=3$ we simply have that $j_{1/2,k} = k\pi$, leading to a right-hand side of $c(k,3) = k^2/(k+1)^2$ for both even and odd $k$, and thus (2) in this case.
 
 For higher dimensions, using McMahon's asymptotic expansions for large zeros of $J_{\nu,k}$ with fixed $\nu$~\cite[10.21(vi)]{olvemaxi}, namely
 \[
  j_{\nu,k} = \left( k + \fr{\nu}{2}-\fr{1}{4}\right)\pi + \bo\left(k^{-1}\right), \mbox{ as } k\to\infty,
 \]
we obtain the desired limits (2).
 
To show that $c(k,d)<1$ whenever $d \geq 4$, we use the following inequality~\cite{infsaf}
 \[
  j_{\nu,k} > \nu + k \pi -\fr{1}{2}, \qquad \nu>1/2, k\in\N
 \]
to obtain, for odd $k$,
\begin{equation}
\label{eq:odd-k-bound}
	c(k,d)^{1/2} = \fr{\pi k}{2 j_{d/2-1,(k+1)/2}} \leq  \fr{\pi k}{2\left( \fr{d}{2}-1 + \fr{k+1}{2}\pi -\fr{1}{2}\right)} = \fr{\pi k}{d-2 + (k+1)\pi -1}<1,
\end{equation}
and, similarly, for even $k$,
\begin{equation}
\label{eq:even-k-bound}
\begin{aligned}
	c(k,d)^{1/2} &= \fr{\pi k}{j_{d/2-1,k/2}+j_{d/2-1,k/2+1}}\\ &\leq  \fr{\pi k}{\left(\fr{d}{2}-1+\fr{k}{2}\pi-\fr{1}{2}\right)
	+ \left(\fr{d}{2}-1+\left(\fr{k}{2}+1\right)\pi-\fr{1}{2}\right) } = \fr{\pi k}{d-3 + (k+1)\pi}<1.
\end{aligned}
\end{equation}
Finally, we note that (3) also follows directly from the bounds \eqref{eq:odd-k-bound} and \eqref{eq:even-k-bound}.
\end{proof}

Obtaining sharp lower bounds for $\alpha_{k,d}$ in a way similar to what was done in Section~\ref{sec:k=1} for $\alpha_{1,d}$ would depend on having sharp lower bounds for
higher eigenvalues of the same type as the Payne--Weinberger bound. Since no such bounds are available, the best that can be done along these lines at this stage is simply to
note that $\mu_{k}(\Omega) \geq \mu_{1}(\Omega)$ for all $k$ larger than one, and then proceed in the same way as in the first part of the
proof of Theorem~\ref{scaledmonot1}. Doing this we obtain
\[
 \alpha_{k,2} \geq \fr{\pi^2}{\left( 2j_{0,1}+(k-1)\pi\right)^2},
\]
and similarly for higher dimensions. Due to not having used sharp bounds for the $k^{\rm th}$ eigenvalue, these bounds cannot display the right asymptotic behaviour in 
the order of the eigenvalue. However, fixing $k$ and deriving the corresponding bound depending on the dimension will provide bounds with the correct behaviour in $d$.
For the second nontrivial eigenvalue and $d$ larger than two, for instance, this yields
\[
 \alpha_{2,d} \geq  \fr{\pi^2}{\left( j_{(d-2)/2,1}+j_{(d-2)/2,2}\right)^2}.
\]
\txtb{As N. Filonov has pointed out to us~\cite{filo1}, by combining Theorem~3.2 in~\cite{filo} with an improvement to Funano's argument, it is possible to show that
\[
 \alpha_{k,d} \geq \fr{\pi^2}{16 j_{d/2-1}^2}.
\]
}
% \[
%  \fr{\pi^2}{\left( 2j_{0,1}+\pi\right)^2}\leq \alpha_{2,2} \leq \fr{4\pi^2}{\left( 2j_{0,1}+\pi\right)^2}
% \]

We may also obtain that the coefficients $\alpha_{k,d}$ are monotonically decreasing in the dimension.

\begin{theorem}
\label{thm:dimension-monotonicity}
 We have $\alpha_{k,d}\geq \alpha_{k,d+1}$.
\end{theorem}
\begin{proof}
 Assume there existed $k$ and $d$ such that $\alpha_{k,d} < \alpha_{k,d+1}$. Then there would exist bounded domains $\Omega_{1}\subset\Omega_{2}\subset\R^{d}$ such that
\[
 \fr{\mu_{k}(\Omega_{1})}{\mu_{k}(\Omega_{2})} < \alpha_{k,d+1}.
\]
Consider now the domains $\Omega_{1}\times[0,\ell] \subset \Omega_{2}\times [0,\ell]\subset \R^{d+1}$. The $k^{\rm th}$ eigenvalue of these domains is given by 
\[
 \mu_{k}\left(\Omega_{1}\times[0,\ell]\right) = \mu_{m_{1}}\left(\Omega_{1}\right)+\fr{\pi^2 n_{1}^2}{\ell^2} \quad \mbox{ and } \quad
 \mu_{k}\left(\Omega_{2}\times[0,\ell]\right) = \mu_{m_{2}}\left(\Omega_{1}\right)+\fr{\pi^2 n_{2}^2}{\ell^2},
\]
for some integers $m_{1},n_{1},m_{2},n_{2}\in\N_{0}$. We now pick $\ell$ small enough that $n_{1}=n_{2}=0$, and thus
\[
 \fr{ \mu_{k}\left(\Omega_{1}\times[0,\ell]\right) }{ \mu_{k}\left(\Omega_{2}\times[0,\ell]\right) }=
 \fr{ \mu_{k}\left(\Omega_{1}\right) }{ \mu_{k}\left(\Omega_{2}\right) } < \alpha_{k,d+1},
\]
yielding a contradiction.
\end{proof}
\begin{remark}
 Note that the upper bounds $c(k,d)$ given by~\eqref{eq:inf-kroeger} also satisfy this monotonicity in the dimension, due to the \txtb{monotonicity
 of the Bessel zeros of $j_{\nu,k}$ with respect to the order $\nu$~\cite{boch}}.
\end{remark}

\section{Domain quasi-monotonicity with a fixed inner domain}
\label{sec:polya}

In some particular instances of the smaller domain $\Omega_{1}$, it is possible to ensure monotonicity irrespective of the larger domain $\Omega_{2}$, which may now not even
be convex. This will happen when $\Omega_{1}$ is a maximiser of some $\mu_{k}$ under a volume restriction, as is the case of one ball and of two equal balls for $\mu_{1}$ and
$\mu_{2}$, respectively. This may be seen directly from the maximisation property of $\Omega_{1}$ and the scaling property $\mu_{k}(c \Omega) = c^{-2}\mu_{k}(\Omega)$ -- see~\cite{szeg,wein}
and~\cite{buhe} for the maximising properties of one and two balls, respectively.

We can say more if the smaller domain is fixed and we consider a variable sequence of domains containing it. In fact, in this case no assumptions on
convexity are necessary.
\begin{theorem}
\label{thm:fixed-inner-domain}
 Let $\Omega$ be a domain in $\R^{n}$, and $\Omega_{k}\subset\R^{n}$ a family of bounded domains such that $\Omega\subset\Omega_{k}$ for all $k\in\N$.
 Then
 \begin{equation}
\label{eq:outer-limit}
  \liminf_{k\to\infty} \fr{\mu_{k}(\Omega)}{\mu_{k}(\Omega_{k})} \geq 1.
 \end{equation}
\end{theorem}
\begin{proof}
 We will use the inequalities $\mu_{k}(\Omega_{k})\leq \lambda_{k}(\Omega_{k})\leq\lambda_{k}(\Omega)$,
 where the first inequality is a standard relation between Neumann and Dirichlet eigenvalues, while the
 second is a consequence of the monotonicity by inclusion for Dirichlet eigenvalues. Using this, we obtain
 \[
  \fr{\mu_{k}(\Omega)}{\mu_{k}(\Omega_{k})}\geq \fr{\mu_{k}(\Omega)}{\lambda_{k}(\Omega_{k})}\geq \fr{\mu_{k}(\Omega)}{\lambda_{k}(\Omega)}.
 \]
Since the first term in the Weyl asymptotics is the same for both Dirichlet and Neumann boundary conditions, the result follows.
\end{proof}

\begin{example}
If we fix the outer domain and allow for a sequence of (varying) inner domains, then it is not possible to expect any result of this
type to hold without imposing some restriction on the elements of the sequence. Consider the sequence of domains $\Omega_{j^2-1}$ \txtb{given} by $j^2$ equal (disjoint)
disks of radius $1/j$. These may be placed in such a way that they are contained
in a fixed square $S$ of side-length two, and we thus have $\Omega_{j^2-1}\subset S$
while $\mu_{j^2-1}(\Omega_{j^2-1}) =0$ for all $j$, as each domain $\Omega_{j^2-1}$ has $j^2$ disjoint components. By taking $\Omega_{k}$ to be
a single disk of unit radius when $j$ is not of the form $k^2-1$, for instance, we may thus build a sequence of
domains $\Omega_{k}$ such that $\left|\Omega_{k}\right|=\pi$, $\Omega_{k}\subset S$, and
\[
  \liminf_{k\to\infty} \fr{\mu_{k}(\Omega_{k})}{\mu_{k}(S)} = 0.
\]
\end{example}

\begin{remark}
\label{rem:colels}
If in place of \eqref{eq:outer-limit} we consider the volume-adjusted ratio
\begin{displaymath}
	\frac{\mu_{k}(\Omega)|\Omega|^{2/d}}{\mu_{k}(\Omega_{k})|\Omega_k|^{2/d}}
\end{displaymath}
(where, as before, $\Omega,\Omega_k$ are bounded regular domains with $\Omega$ fixed and $\Omega \subset \Omega_k$ for all $k \in \N$) then the stronger asymptotic behaviour
\begin{equation}
\label{eq:outer-limit-stronger}
	\liminf_{k\to\infty} \frac{\mu_{k}(\Omega)|\Omega|^{2/d}}{\mu_{k}(\Omega_{k})|\Omega_k|^{2/d}} \geq 1
\end{equation}
is actually equivalent to P\'olya's conjecture for the Neumann eigenvalues,
\begin{equation}
\label{eq:polya}
	\mu_k^\ast := \sup\{\mu_{k}(\Omega): \Omega \subset \R^d \text{ bounded, regular, } |\Omega| = 1\} \leq \frac{4\pi^2k^{2/d}}{\omega_d^{2/d}}
\end{equation}
for all $k \in \N$, where $\omega_d$ is the volume of \txtb{a} ball of unit radius in $\R^d$.
\end{remark}

The implication \eqref{eq:outer-limit-stronger} $\implies$ \eqref{eq:polya} is a quite direct consequence of the superadditivity property of the sequence $\mu_k^\ast$ established by Colbois and El Soufi in \cite{colels}: for simplicity we assume that for all $k$ there exists a bounded regular domain $\Omega_k^\ast$ for which $|\Omega_k^\ast|=1$ and $\mu_k(\Omega_k^\ast) = \mu_k^\ast$ (if not, we take a sequence $\varepsilon_k \to 0$ and $\Omega_k^\ast$ such that $\mu_k(\Omega_k^\ast) > \mu_k^\ast - \varepsilon_k$; the following argument then needs to be adjusted by a multiplicative factor converging to one). We fix any (bounded regular) domain $\Omega \subset \R^d$ and, if necessary, translate and rescale $\Omega_k^\ast$ to produce a new domain $\Omega_k$ containing $\Omega$. Then
\begin{displaymath}
	 \fr{ \mu_{k}\left(\Omega\right)}{\mu_{k}\left(\Omega_{k}\right)}
	= \fr{ \mu_{k}\left(\Omega\right) \left|\Omega\right|^{2/d}\omega_{d}^{2/d} }{4\pi^2 k^{2/d}}\times
	\fr{4\pi^2 k^{2/d}}{\omega_{d}^{2/d}\mu_{k}^{\ast}} \times \left(\fr{\left|\Omega_{k}\right|}{\left|\Omega\right|}\right)^{2/d}.
\end{displaymath}
Using that the first factor on the right-hand side converges to one by the Weyl asymptotics applied to $\Omega$, and using \eqref{eq:outer-limit-stronger}, we obtain
\begin{displaymath}
	\liminf_{k\to\infty} \fr{4\pi^2 k^{2/d}}{\omega_{d}^{2/d}\mu_{k}^{\ast}} \geq 1.
\end{displaymath}
By \cite[Corollary~2.2]{colels} \txtb{(in the Neumann case) the limit exists, equals the infimum of the sequence, and} this is equivalent to~\eqref{eq:polya}.

For the implication \eqref{eq:polya} $\implies$ \eqref{eq:outer-limit-stronger}, which does not require the assumption $\Omega \subset \Omega_k$, with $\mu_k^\ast$ as in \eqref{eq:polya}, mimicking the above argument, we have
\begin{displaymath}
\begin{array}{lll}
  \fr{ \mu_{k}\left(\Omega\right)}{\mu_{k}\left(\Omega_{k}\right)} & \geq & \fr{ \mu_{k}\left(\Omega\right) \left|\Omega_{k}\right|^{2/d} }{\mu_{k}^{*}}\eqskip
  & = & \fr{ \mu_{k}\left(\Omega\right) \left|\Omega\right|^{2/d}\omega_{d}^{2/d} }{4\pi^2 k^{2/d}}\times\fr{4\pi^2 k^{2/d}}{\omega_{d}^{2/d}\mu_{k}^{*}}
  \times \left(\fr{\left|\Omega_{k}\right|}{\left|\Omega\right|}\right)^{2/d}.%\eqskip
%  & \geq & \fr{ \mu_{k}\left(\Omega\right) \left|\Omega\right|^{2/d}\omega_{d}^{2/d} }{4\pi^2 k^{2/d}}
%  \times\fr{4\pi^2 k^{2/d}}{\omega_{d}^{2/d}\mu_{k}^{*}}
%  \times \left(\fr{\omega_{\infty}^{-}}{\left|\Omega\right|}\right)^{2/d} %\eqskip
  %& 
%  \to 
  %& 
%  \left(\fr{\omega_{\infty}^{-}}{\left|\Omega\right|}\right)^{2/d}
\end{array}
\end{displaymath}
The first factor on the right-hand side again converges to one, and P\'olya's conjecture implies that the second term is always at least one.

\section{Open questions and problems}
\label{sec:open}

We finish by collecting a number of open questions involving $\optconst{k}{d}$ given by \eqref{eq:inf}, and related quantities. We first observe that, \textit{a priori},
unless the upper bounds given in Section~\ref{sec:asymptotics} coming from Kr\"{o}ger's eigenvalue bounds should happen to give the correct values of $\optconst{k}{d}$
for more general $k$, there is no reason to expect these values (and, correspondingly, the optimising pairs of domains) to be explicitly computable in terms of known
constants and quantities. By way of analogy, the numerically obtained maximisers for $\mu_k (\Omega)$ (and minimisers for $\lambda_k (\Omega)$) among all sufficiently
regular domains $\Omega$ in dimension $2$ \cite{af12,ao17}, except for a few very small values of $k$, do not in general seem to correspond to any explicitly describable known
domains; indeed, establishing any analytic properties of these domains is a very hard problem.

\begin{problem}
\label{question:optima}
Is $\optconst{k}{d}$ attained as a degenerate limit by domains $\Omega_{2,n}$ collapsing to a line segment $\Omega_1$? If so, do the upper bounds in~\eqref{eq:inf-kroeger}
actually give the true values of $\optconst{k}{d}$? Alternatively, for some $k,d$, is $\optconst{k}{d}$ attained as a degenerate limit by two different sequences of domains
$\Omega_{1,n}$ and $\Omega_{2,n}$ collapsing to a line segment in different ways (and thus approaching different Sturm--Liouville problems)? Can these domains and limit
problems be identified, at least in some cases beyond $k=1$?
\end{problem}

There are several other natural questions which may be more approachable. For example, the asymptotic behaviour of the eigenvalues of a fixed domain, or of a fixed inner domain
and variable outer domain (Proposition~\ref{static} and
Theorem~\ref{thm:fixed-inner-domain}), as well as the nature and behaviour of the upper bounds in Section~\ref{sec:asymptotics}, provide circumstantial evidence
in favour of the conjecture that $\optconst{k}{d} \to 1$.

\begin{problem}
\label{question:asymptotics}
Is it true that, for fixed dimension $d \geq 2$, $\optconst{k}{d} \to 1$ as $k \to \infty$?
\end{problem}

Related is the following:

\begin{problem}
For fixed dimension $d \geq 2$, is $\optconst{k}{d}$ monotonically increasing in $k\in\N$?
\end{problem}

Switching from varying $k$ to varying $d$, we note that Funano's result explicitly excludes the trivial case $d=1$, since otherwise the constant $C$
in~\eqref{eq:inf-funano} could not exceed $1$; the value that we obtain on excluding $d=1$ is clearly larger (cf.\ \eqref{eq:kroeger-funano}), and,
likewise, the bounds we obtain in Section~\ref{sec:asymptotics} (after normalisation by $d^2$) increase with $d$, and in fact converge to $\pi^2$ as
$d \to \infty$. Thus one may ask:

\begin{problem}
Is it true that, for fixed $k \in \N$, the dimensionally normalised values $\optconst{k}{d}d^2$ are increasing (but bounded) functions of $d \geq 2$?
\end{problem}

Finally, as discussed in Section~\ref{sec:polya}, it seems that in some sense the smaller domain $\Omega_1$ in the inclusion $\Omega_1 \subset \Omega_2$ plays
a more important role than $\Omega_2$. 

\begin{problem}
Define
\begin{displaymath}
	\genoptconst{k}{d} := \inf \left\{\frac{\mu_k (\Omega_1)}{\mu_k (\Omega_2)} : \Omega_1\subset\Omega_2
	\subset \R^d \text{ bounded Lipschitz domains, $\Omega_1$ convex}\right\},
\end{displaymath}
for $k \in \N$ and $d\geq 2$, so that $\genoptconst{k}{d} \leq \optconst{k}{d}$. Do we still have $\genoptconst{k}{d}>0$? Do~\eqref{eq:inf-funano} and~\eqref{eq:k=1}
(for $k=1$) hold with $\genoptconst{k}{d}$ in place of $\optconst{k}{d}$?
\end{problem}

\subsection*{Acknowledgements.} The authors wish to express their thanks to Kei Funano and Nikolay Filonov for several helpful comments related to improved bounds for the constants $\alpha_{k,d}$ considered in the paper. The work of both authors was partly supported by the Funda\c{c}\~ao para a Ci\^encia e a Tecnologia, Portugal, via the research centre GFM, reference UID/00208/2023 (both authors), the research centre CIDMA, reference UID/04106/2023 (JBK) and the project SpectralOPs, reference \href{https://doi.org/10.54499/2023.13921.PEX}{2023.13921.PEX} JBK.


\begin{thebibliography}{99}

% \bibitem{ap}
% P.R.S. Antunes, private communication.

\bibitem{af12}
P.R.S. Antunes and P. Freitas,
\textit{Numerical optimisation of low eigenvalues of the Dirichlet and Neumann Laplacians},
J. Optim. Theory Appl. {\bf 154} (2012), 235--257.

\bibitem{ao17}
P.R.S.~Antunes and E.~Oudet, 
Numerical results for extremal problem for eigenvalues of the Laplacian, pp 398--411 in 
A.~Henrot (ed.), Shape Optimization and Spectral Theory, De Gruyter Open, Warsaw-Berlin, 2017.

\bibitem{bebe}
M. Bebendorf,
\emph{A note on the Poincar\'{e} inequality for convex domains},
Z. Anal. Anwend. {\bf 22} (2003), 751--756.

\bibitem{boch}
\txtb{
M. B\"{o}cher, \emph{On certain methods of Sturm and their application to the roots of Bessel's functions},
Bull. Amer Math. Soc. {\bf 3} (1897), 205--213.
}

\bibitem{buhe}
D. Bucur and A. Henrot, 
\emph{Maximization of the second non-trivial Neumann eigenvalue},
Acta Math. {\bf 222} (2019), 337--361.

\bibitem{chav}
I. Chavel,
Eigenvalues in Riemannian geometry, Pure and Applied Mathematics {\bf 115}.
Academic Press, Inc., Orlando, FL, 1984.

\bibitem{colels}
B.~Colbois and A.~El~Soufi,
\emph{Extremal eigenvalues of the Laplacian on Euclidean domains and closed surfaces},
Math.\ Z.\ {\bf 278} (2014), 529--549.

\bibitem{filo}
\txtb{
N. Filonov, \emph{On the P\'{o}lya conjecture for the Neumann problem in planar convex domains},
\txtb{ Comm. Pure Appl. Math. {\bf 78} (2025), 537--544.}
}

\bibitem{filo1}
\txtb{
N. Filonov, private communication. September 11, 2023.
}

\bibitem{f16}
K.~Funano,
\emph{Applications of the `Ham Sandwich Theorem' to eigenvalues of the {L}aplacian},
Anal.\ Geom.\ Metr.\ Spaces \textbf{4} (2016), 317--325.

\bibitem{f22}
K.~Funano,
\emph{A note on domain monotonicity for the {N}eumann eigenvalues of the {L}aplacian},
\txtb{Illinois J. Math. {\bf 67}  (2023), 677--686}.

%\bibitem{j00}
%D.~Jerison, \emph{Locating the first nodal line in the Neumann problem}, Trans.\ Amer.\ Math.\ Soc.\ \textbf{352} (2000), 2301--2317.

\bibitem{hm}
A, Henrot and M. Micchetti,
\emph{Optimal bounds for Neumann eigenvalues in terms of the diameter},
\txtb{Ann. Math. Qué. {\bf 48} (2024), 277--308.}


% \bibitem{heth}
% H.W. Hethcote,
% \emph{Bounds for zeros os some special functions},
% Proc.\ Amer.\ Math.\ Soc.\ {\bf 25} (1970), 72--74.

\bibitem{infsaf}
E.K.~Ifantis and P.D.~Siafarikas,
\emph{A differential equation for the zeros of Bessel functions},
Applicable Anal.\ {\bf 20} (1985), 269--281. 

\bibitem{k92}
P.~Kr\"{o}ger,
\emph{Upper bounds for the Neumann eigenvalues on a bounded domain in Euclidean space},
J.\ Funct.\ Anal.\ {\bf 106} (1992), 353--357.

\bibitem{k99}
P.~Kr\"{o}ger, 
\emph{On upper bounds for high order {N}eumann eigenvalues of convex domains in {E}uclidean space}, 
Proc.\ Amer.\ Math.\ Soc.\ \textbf{127} (1999), 1665--1669.

\bibitem{mccartin}
B.J.~McCartin, \emph{Eigenstructure of the equilateral triangle, Part II: The Neumann problem}, 
Mathematical Problems in Engineering {\bf 8} (2002), 517--539.

%\bibitem{ni20}
%C. P. Niculescu and I. Roventa,
%\emph{Convex functions and Fourier coefficients}, Positivity\ \textbf{24} (2020), 129--139.

\bibitem{olvemaxi}
F.W.J.~Olver and L.C.~Maximon, Bessel functions, in 
F.W.J.~Olver, D.W.~Lozier, R.F.~Boisvert and C.W.~Clark (eds.), NIST Handbook of Mathematical Functions, Cambridge University Press, New York, NY, 2010.

\bibitem{pw60}
L.~Payne and H.~Weinberger, \emph{An optimal {P}oincar\'e inequality for convex domains}, Arch.\ Rational Mech.\ Anal.\ \textbf{5} (1960), 286--292.

\bibitem{szeg}
G. Szeg\H{o}, \emph{Inequalities for certain eigenvalues of a membrane of given area.}  
J. Rational Mech. Anal. 3, (1954). 343--356.

\bibitem{wats}
G.N. Watson, A treatise on the theory of Bessel functions. Cambridge: Cambridge University
Press 1944; New York: MacMillan 1944.

\bibitem{wein}
H. F. Weinberger, \emph{An isoperimetric inequality for the $N$-dimensional free membrane problem.} J. Rational Mech. Anal. 5 (1956), 633--636.
\end{thebibliography}
\end{document}